\documentclass{amsart}
\usepackage{amsmath,amssymb,amscd,amsthm}
 \usepackage{color}%

\def\nt{\noindent}
\newtheorem{lemma}{Lemma}[section]
\newtheorem{theorem}[lemma]{Theorem}
\newtheorem{proposition}[lemma]{Proposition}
\newtheorem{corollary}[lemma]{Corollary}
\theoremstyle{definition}
\newtheorem{definition}[lemma]{Definition}
\newtheorem{remark}[lemma]{Remark}

\newtheorem{example}[lemma]{Example}
\newtheorem{examples}[lemma]{Examples}
\newtheorem{question}[lemma]{Question}

\newcommand\ec { \color{black}}%

\def\SQ{\mathbb Q} 
\def\SZ{\mathbb Z} 
\def\SF{\mathbb F} 
\def\FF{\mathfrak{F}} 
\def\F{\textbf{F}} 

\def\f{\mathbf{f} }
\def\int{\operatorname{Int}} 
\def\spec{\operatorname{Spec}} 
\def\max{\operatorname{Max}} 
\def\tmax{\operatorname{\textit{t}-max}} 
\def\r{\overline{R} }
\def\ii{(I \colon I)} 
\def\m{{\mathcal M}} 
\def\int{\operatorname{Int}}    
\def\intz{\operatorname{\int(\SZ)}}    
\begin{document}
\title[Flat Ideals and Stability in Integral Domains]
{Flat Ideals and Stability in Integral Domains}
\author{Giampaolo PICOZZA}
\address{Universit\'e Paul C\'ezanne,
Facult\'e des Sciences et Techniques, 13397 Marseille
Cedex 20, France} \email {giampaolo.picozza@univ-cezanne.fr}
\author{Francesca Tartarone}
\address{Universit\`{a}
degli studi Roma Tre, Dipartimento di Matematica,  Largo San
Leonardo Murialdo 1, 00146 Roma, Italy}
\email{tfrance@mat.uniroma3.it}

\date{\today}
\begin{abstract} We introduce the concept of \textit{quasi-stable} ideal in an
integral domain $D$ (a nonzero fractional ideal $I$ of $D$ is
quasi-stable if it is flat in its endomorphism ring $(I \colon I)$)
and study properties of domains in which each nonzero fractional
ideal is quasi-stable. We investigate some questions about flatness
that were raised by S. Glaz and W.V. Vasconcelos in their 1977 paper
\cite{GV}.
\end{abstract}

\maketitle
\section*{Introduction}

Throughout
  the paper $D$ is an integral domain with quotient field $K$, an ideal is a fractional ideal and an integral ideal is
  an ideal contained in $D$.

 The property of flatness for ideals in commutative rings has been
 investigated in many interesting papers. We recall some of them
 that inspired part of this work: J.D. Sally \& W.V. Vasconcelos
 \cite{flat1} (1975), S. Glaz \& W.V. Vasconcelos \cite{GV, flat3} (1977, 1984), D.D. Anderson \cite{and1}(1983) and M.
 Zafrullah \cite{zaflat} (1990).

 More recently many researchers have studied ideals which satisfy the
 following stability criterion:
a nonzero   ideal $I$ of $D$ is \textit{stable} if $I$ is projective
in the endomorphism ring $(I \colon I)$ and a domain $D$ is
\textit{stable}
  if each nonzero   ideal of $D$ is stable ($D$ is
  \textit{finitely stable} if each nonzero finitely generated   ideal  of $D$ is stable).
In particular, stable ideals and domains
  have been widely investigated by D.E. Rush \cite{rush} (1995), B.
Olberding \cite{olb2, O, olb} (1998, 2001, 2002) and H.P. Goeters
\cite{goeters} (1998).
 Some aspects of their work on stability have been also
deepened  by studying   properties of the class semigroup of $D$
such as the Clifford regularity (Cf. S. Bazzoni \cite{baz1,baz2}
(2000, 2001)). Moreover, in \cite{km, km2} S.E. Kabbaj \& A. Mimouni
have strengthened the notion of stable ideal (and domain)
considering the so called \textit{strongly stable ideals}, that is
nonzero   ideals which are  principal in their endomorphism ring
(analogously, a domain $D$ is \textit{strongly stable} if each
nonzero  ideal of $D$ is strongly stable).

In integral domains   the   properties of being projective and
invertible for an ideal $I$ are equivalent (analogously, free is
equivalent to principal), and flatness is a natural generalization
of the projective property. In \cite{goeters} the condition that a
nonzero   ideal $I$ is flat in $(I \colon I)$ is investigated in
Noetherian domains and it is shown that,  if $D$ is Noetherian, this
property holds for each nonzero fractional ideal of $D$ if and only
if $D$ is stable.

In this paper  we attempt to  link the two concepts of flatness and
stability for ideals in integral domains, by considering  {\it
quasi-stable} ideals: a nonzero   ideal $I$ is quasi-stable if it is
flat in $(I \colon I)$. So, the quasi-stable property generalizes
 the stable property  (instead of strengthening it as in
\cite{km}). The study of quasi-stable ideals   has required a more
general investigation on flatness of ideals which turned out to be
useful to deepen some open problems.

Whether flat ideals of integrally closed domains are complete is a
question that has been first posed in \cite{flat1}. In that paper
(and in the following \cite{GV}) the authors address the
divisibility problem for flat ideals, that is, the problem of
deciding when an element belongs to a flat ideal. One of the main
tools in this study is what they called ``the divisibility lemma'',
which is, in modern language, the fact that a flat ideal is a
$w$-ideal. In the introduction of \cite{GV}, the authors say that
the last section of that article ``contains a number of unresolved
questions where the elusive completeness of flat ideals plays a
significant role'' and they add later in the paper that
``unfortunately other that the few cases of \cite{flat1}, not much
seems known'' (the cases are those of Krull domains, GCD-domains and
integrally closed coherent domains, Cf. \cite[Example 1.5]{flat1}).
In Section \ref{sec:flat} we improve the divisibility lemma (Theorem
\ref{thm:tflat}), by showing that a flat ideal is not only a
$w$-ideal, but it is in fact a $t$-ideal, and obtain, by using some
well-known properties of star operations, the completeness of flat
ideals in integrally closed domains.

Another question considered   in   \cite{flat1} and \cite{GV} is
related to the characterization  of domains in which flat ideals are
finitely generated  (and so, invertible). For example, in
\cite[Theorem 3.1]{flat1} it is shown that a flat ideal of a
polynomial ring with finitely generated content is invertible. It is
also observed that flat ideals in Krull domains are invertible. In
\cite[\S 3]{GV} it is conjectured that faithfully flat ideals in
H-domains are invertible (an H-domain is a domain in which every
$t$-maximal ideal is divisorial). We show that this is not true, by
giving a counterexample (Example \ref{conjecture2}). On the other
side, we show that the $t$-finite character on $D$ suffices to have
that all faithfully flat ideals are invertible (Proposition
\ref{ff-t-finite character}).   This result may be related to the
Bazzoni's conjecture \cite{bazzoni}, recently proven in \cite{hmmt}
and in \cite{hk}, which states that all locally invertible (i.e.,
faithfully flat) ideals of a Pr\"ufer domain are invertible if and
only if the domain has the ($t$-)finite character on maximal ideals.

In Section \ref{sec:quasi-stable}, with the necessary assumption of
the $t$-finite character, we characterize stable domains as the
domains in which each ideal is faithfully flat in its endomorphism
ring (Proposition \ref{t-finite character}). So, it seems natural to
define a new class of domains,  the \emph{quasi-stable domains},
that is, the domains such that each nonzero ideal is flat in its
endomorphism ring. We show that this class is strictly larger than
the class of stable domains (this is easy to see) and, with an
elaborate example, that it is smaller than the class of finitely
stable domains, even if these two classes coincide for Noetherian
and integrally closed domains.

In Section \ref{sec:overrings}, we study overrings and localizations
of quasi-stable domains and   show that they are still quasi-stable
  in some significant cases.

\section{Flat ideals and $t$-ideals} \label{sec:flat}

We recall some basic terminology and notions about divisorial
ideals, $t$-ideals and $w$-ideals. Given a domain $D$ with quotient
field $K$, we put $\FF(D)$ to be the set of nonzero $D$-modules
contained in $K$, $\f(D)$  the set of nonzero finitely generated
$D$-modules contained in $K$ and
  $\F(D)$ the set of nonzero fractional ideals of $D$.

If $I$ is a nonzero ideal of $D$, then:

\begin{itemize}
\item the \textit{divisorial closure} of $I$ is the ideal $I_v := (D \colon (D \colon
I))$, where $(D \colon H) := H^{-1}:= \{x \in K \mid xH \subseteq
D\}$, for each $H \in \F(D)$ ;

\item the \textit{$t$-closure} of $I$ is the ideal
$I_t := \bigcup_{J \in \f(D), \, J \subseteq I} J_v.$

\item  the {\it $w$-closure} of $I$ is the ideal $I_w :=
\bigcup_{J \in \f(D),  \, J_v=D}(I \colon J).$

\end{itemize}

 An ideal $I \in \F(D)$ is \textit{divisorial} (respectively, a $t$-ideal or a $w$-ideal) if
 $I=I_v$ (respectively, $I=I_t$ or $I=I_w$). For each $I \in \F(D)$, the following inclusions hold: $I \subseteq
I_w \subseteq I_t \subseteq I_v$.

An ideal $I$ is \textit{$t$-finite} if there exists a finitely
generated ideal $J \subseteq I$ such that $J_t = I_t$.

The $v$-, $t$- and $w$-operations are particular
\textit{star-operations} (see, for instance, \cite{om, fh}). The
$t$-operation is a \textit{star-operation of finite type}, that is,
for each $H \in \F(D)$:
$$H_t:=\bigcup \{F_t \mid F \subseteq H, \,\, F \in
\f(D)\}.$$

Moreover, $t$ is   maximal among the star-operations of finite type
on $D$  that is, if $\star$ is a finite type star-operation on $D$,
then $\star \leq t$ (i.e., $H_{\star} \subseteq H_t$, for each $H
\in \F(D)$).

An ideal of a domain $D$ is flat if it is flat as a $D$-module. A
useful characterization of flat ideals in integral domains is the
following (\cite[Theorem 2]{and1}):

\begin{proposition} \label{char:flatness} Let  $D$ be an integral domain. An ideal $I$ of $D$ is flat if and
only if $(A \cap B)I = AI \cap BI$ for all   $A,B \in \F(D)$.
\end{proposition}

Being projective, invertible ideals are flat. We give a short proof
of this fact, by using the previous characterization. Note that it
is always true that, if $A,B$ and $C$ are (fractional) ideals of
$D$, then $C(A \cap B) \subseteq CA \cap CB$. So, let $I$ be
invertible and $A$ and $B$ ideals of $D$. Then: $$IA \cap IB =
II^{-1}(IA \cap IB) \subseteq I(I^{-1}IA \cap I^{-1}IB)= I(A \cap
B).$$ Thus $I$ is flat.

Note that flat ideals are not always invertible. For example, we
recall that Pr\"ufer domains are exactly the domains in which each
ideal is flat (\cite[Theorem 25.2]{gilmer}  and Proposition
\ref{char:flatness}). So, in a non-Dedekind Pr\"ufer domain, each
non finitely generated ideal is flat but not invertible (we can take
$D := \textrm{Int}(\SZ) :=\{f(X) \in \SQ[X] \mid f(\SZ) \subseteq
\SZ \}$, \cite[\S \, 6]{cc}.)

On the contrary, it is well-known that  even in the more general
context of rings with zero divisors  finitely generated ideals are
flat if and only if they are projective. So, in a domain, finitely
generated flat ideals are invertible. More precisely we have the
following (Cf. \cite[Proposition 1]{zaflat}):

\begin{proposition}\label{prop:flat_inv} Let $D$ be an integral domain and $I$ a $t$-finite  ideal of $D$.
 Then $I$ is flat if and only if it is
invertible.
\end{proposition}

\begin{proof}
We have already shown that invertible ideals are flat. So, let $I$
be a $t$-finite ideal. Then, there exists an ideal $J=(a_1, a_2,
\ldots, a_n)$, $J \subseteq I,$    such that $J_t = I_t$
(and so $I^{-1}=J^{-1}$). We have that:

\begin{align*}
 D \supseteq II^{-1} = IJ^{-1}&= I(a_1^{-1}D \cap a_2^{-1}D \cap
\ldots \cap a_n^{-1}D) \\    &= (Ia_1^{-1}D \cap Ia_2^{-1}D \cap
\ldots \cap Ia_n^{-1}D) \supseteq D,
\end{align*}

 \nt where the third equality holds for the flatness of $I$ over $D$. Thus, $II^{-1}=D$ and $I$ is invertible. \end{proof}

A consequence of this fact is that in   Krull and Noetherian domains
(and more in general in Mori domains), the flat ideals are exactly
the invertible ideals ({\cite[Corollary 4]{zaflat}).

It is known that flat ideals are $w$-ideals (or semidivisorial
ideals, in the language of Glaz \& Vasconcelos, \cite[Corollary
2.3]{GV}). We can show that flat ideals are in fact $t$-ideals. We
will use the following  lemma.

\begin{lemma} \label{lemma:ij-1} Let $D$ be an integral domain, $J$ a   nonzero  finitely generated ideal of $D$.
If $I$ is a flat ideal of $D$, then $(I:J)=IJ^{-1}$.
\end{lemma}
\begin{proof}
Let $J = (a_1, a_2, \ldots, a_n)$. Then, by the flatness of $I$, we
have that: $$(I:J)=(a_1^{-1}I \cap a_2^{-1}I \cap \ldots \cap
a_n^{-1}I)= I(a_1^{-1}D \cap a_2^{-1}D \cap \ldots \cap a_n^{-1}D) =
IJ^{-1}.$$ \end{proof}

\begin{theorem} \label{thm:tflat}
Let $D$ be an integral domain and $I$ be a nonzero ideal of $D$. If $I$
is flat then $I$ is a $t$-ideal.
\end{theorem}

\begin{proof}
Let $J$ be a nonzero finitely generated ideal. Then, since $I$ is
flat, $(I \colon J) = IJ^{-1}$ (by Lemma \ref{lemma:ij-1}). Now,
$J^{-1} = (J_v)^{-1}$, hence:

$$(I \colon J) = IJ^{-1} = I(J_v)^{-1}= I \bigcap_{x \in J_v}\frac{1}{x}D \subseteq
\bigcap_{x \in J_v}\frac{1}{x}I = (I \colon J_v) \subseteq (I
\colon J).$$

Thus $(I \colon J) = (I \colon J_v)$. If $J \subseteq I$, then $1
\in (I \colon J)$. So, $1 \in (I \colon J_v)$, that is $J_v
\subseteq I$. Hence $I = I_t$. \end{proof}

\begin{remark} \label{remark1}
(1) A divisorial ideal (and so a $t$-ideal) is not always flat. For
instance, take a non-integrally closed domain $D$ in which each
ideal is divisorial (e.g.,  a pseudo-valuation, non valuation,
domain such that the associated valuation domain is two-generated as
a $D$-module  \cite[Corollary 1.8]{pvd2}). Then $D$ has, at least, a
nonzero ideal which is not flat,
 otherwise $D$ would be a  a valuation domain.

(2) Note that prime flat   $t$-ideals  are \textit{well-behaved} in
the sense of Zafrullah (a prime $t$-ideal $P$ of  $D$ is
well-behaved if $PD_P$ is a $t$-ideal in $D_P$ \cite{zaf}). This
follows from the fact that, for ideals, flat implies locally flat,
and flat implies $t$-ideal. So, a prime $t$-ideal which is not
well-behaved is not flat.

(3) In \cite[Proposition 10]{zaflat}, M. Zafrullah has shown that
the integral domains in which each $t$-ideal is flat are precisely
the generalized GCD domains (G-GCD domains) defined in \cite{aa},
that is the domains in which each $t$-finite $t$-ideal is
invertible.
\end{remark}

An immediate corollary of Theorem \ref{thm:tflat} is that in the
statement of Lemma \ref{lemma:ij-1} $J$ can be taken $t$-finite
instead of finitely generated.

\begin{corollary} Let $D$ be an integral domain and $J$ be a nonzero
$t$-finite ideal of $D$.
If $I$ is a flat ideal of $D$ then $(I:J)=IJ^{-1}$.
\end{corollary}

\begin{proof}
Let $H$ be a finitely generated ideal of $D$ such that $H_t= J_t$.
By Theorem \ref{thm:tflat} and Lemma \ref{lemma:ij-1} it follows
that: $$(I:J) = (I_t:J_t) = (I_t:H_t) = (I:H) = IH^{-1} = IJ^{-1}.$$
\end{proof}

We recall  that for each $I \in \F(D)$,  the $b$-closure of $I$ is
defined as follows:
$$I^b := \bigcap IV_\alpha,$$

\nt where the intersection is taken over all valuation overrings
$V_{\alpha}$ of $D$. An ideal $I$ is called \emph{complete} if it is
a $b$-ideal, that is, if $I^b = I$ (\cite[\S \, 24]{gilmer}). As
shown in \cite[Appendix 4, Theorem 1]{zariski-samuel}, the
$b$-closure of an ideal of $D$ coincides with the integral closure
of $I$ in $K$.
By \cite[Appendix 4, Theorem 1]{zariski-samuel} and the definition
of integral dependence and integral closure it follows easily that,
if $D$ is integrally closed, the $b$-operation is a star-operation
and it is of finite type. If $D$ is not integrally closed, the
$b$-closure can be still defined as above for each $I \in \F(D)$,
but in this case it is not a star-operation; it is actually a
semistar operation, which is a generalization of star-operation,
that we don't need to discuss in this context.

In \cite[Conjecture, p.16]{GV}, the authors conjecture that a flat
ideal of an integrally closed domain is complete.

\begin{theorem} \label{conj1} (Cf. \cite[Conjecture, p.16]{GV})
Every flat ideal of an integrally closed domain is complete.
\end{theorem}

\begin{proof} Let $D$ be an integrally closed domain.
As remarked above,  the $b$-operation on $D$ is a star operation of
finite type,
 so  $b \leq t$, that is, $I^b \subseteq I_t$,
for each $I \in \F(D)$. Thus $t$-ideals are complete. From Theorem
\ref{thm:tflat}, flat ideals are $t$-ideals, whence they are
complete.
\end{proof}

In \cite[p.16]{GV}  the authors prove that \textit{if $A$ is an
integrally closed domain of characteristic $2$, then an idempotent
flat ideal of $A$ is a radical ideal.} By using Theorem \ref{conj1},
we can prove this result in any characteristic.

\begin{proposition}
Let $D$ be an integrally closed domain. Then, an idempotent flat
ideal of $D$ is a radical ideal.
\end{proposition}

\begin{proof}
Let $I$ be a flat, idempotent ideal of  $D$. By hypothesis, $D =
\bigcap_{\alpha \in A}V_{\alpha}$, where $\{V_{\alpha}\}_{\alpha \in
A}$ are all the valuation overrings of $D$. Then, for each
 $\alpha \in A$, $IV_\alpha$ is idempotent and so prime
(\cite[Theorem 17.1]{gilmer}). Let $IV_\alpha = P_\alpha$. Since $I$
is flat, then $I$ is complete (Theorem  \ref{conj1}) and so $I =
\bigcap_{\alpha \in A} IV_\alpha = \bigcap_{\alpha \in A} P_\alpha =
\bigcap_{\alpha \in A} (P_\alpha \cap D)$ is an intersection of
prime ideals. Thus $I$ is a radical ideal.
\end{proof}

\begin{remark}
 Note that if all ideals of a domain $D$ are complete then $D$
is a Pr\"ufer domain (\cite[Theorem 24.7]{gilmer}) and so all ideals
are flat (\cite[Theorem 25.2 (c)]{gilmer}). In general, it is not
always true that complete ideals are flat. For instance,   a prime
ideal $P$ of an integrally closed domain $D$ is always complete
since there  always exists a valuation overring of $D$ centered on
$P$ (\cite[Theorem 19.6]{gilmer}).  But, obviously, $P$ is not
always a $t$-ideal thus, in particular, it is not always flat. Such
an example is given by an height-$2$ prime ideal of $\mathbb{Z}[X]$.
In fact, since $\SZ[X]$ is a Krull domain, it is well-known that the
only prime $t$-ideals are the height-one primes.

\end{remark}

We recall that a domain $D$ is an \textit{H-domain} if for each
ideal $I$ of $D$ such that $I^{-1} =D$, there exists a finitely
generated ideal $J \subseteq I$ such that $J^{-1}=  D$. In
\cite[Proposition 2.4]{hz} it is shown that this is equivalent to
the fact that   each $t$-maximal ideal of $D$ (i.e., an ideal  which
is maximal in the set of $t$-ideals of $D$) is divisorial.

In \cite[Proposition 1.1]{GV} it is shown that an ideal $I$ of a
domain $D$ is faithfully flat (as a $D$-module) if and only if it is
flat and locally finitely generated. This is equivalent to saying
that $I$ is faithfully flat if and only if it  is locally invertible
(\cite[Theorem 8]{aa2}).

A second conjecture  stated in \cite[p.9]{GV} is the following: \ec

\medskip

\textit{Conjecture 2 (Cf. \cite[Conjecture, p.9]{GV}): A faithfully
flat ideal in an H-domain is finitely generated.}

\medskip

Now, we give a counterexample showing that this conjecture is false.

\begin{example}\label{conjecture2} We recall that    generalized
Dedekind domains (see, for instance, \cite{gabelli, popescu}) are
examples of H-domains (since   their prime ideals are divisorial,
\cite[Theorem 15]{gabelli}). Now, consider the domain $D := \SZ
+X\SQ[[X]]$. In \cite[Example 2]{gabelli} it is shown that $D$ is
generalized Dedekind. Let $I$ be the ideal of $D$ generated by the
set $\{\frac{1}{p}X \mid p \in \SZ \}$. It is easy to check that $I$
is locally principal. Moreover, in \cite{gabelli} it is also shown
that $I$ is not divisorial. Then $I$ is not finitely generated,
otherwise it would be invertible and so divisorial.
\end{example}

\begin{remark} \label{conjecture-Bazzoni} Conjecture 2 may be refuted also by using
the following argument.   R. Gilmer (\cite[Lemma 37.3]{gilmer}) has
shown that

\smallskip

\begin{lemma} \label{lemma:gilmer} If $D$ is a Pr\"ufer domain with the finite character (i.e.,
each nonzero element of $D$ is contained in finitely many maximal
ideals), then every locally principal ideal (i.e., faithfully flat
ideal)
 of $D$ is invertible. \end{lemma}

\smallskip
In \cite[p.630]{bazzoni} S. Bazzoni   conjectured that:
\medskip

\textit{``Let $D$ be a Pr\"ufer domain. Then every locally principal
ideal
 of $D$ is invertible if and only
if $D$ has the finite character"}

\medskip
\nt and proved this conjecture for some particular Pr\"ufer domains
   (\cite[Theorem 4.3]{bazzoni}). Recently this
conjecture has been
 proven by W.C. Holland, J. Martinez, W.Wm. McGovern, M. Tesemma
(\cite{hmmt}) and, independently, by F. Halter-Koch (\cite{hk}).

\nt In Example~\ref{conjecture2} we have recalled  that generalized
Dedekind domains are H-domains. Now, if Conjecture 2 were true, a
generalized Dedekind domain $D$ would be a Pr\"ufer domain in which
each locally principal ideal is invertible. Hence $D$ would have the
finite character. But the domain $\SZ+X\SQ[[X]]$ considered in
Example \ref{conjecture2} is generalized Dedekind without  the
finite character (the element $X$ is contained in infinitely many
maximal ideals).
\end{remark}

 In \cite{GV} the authors have shown that to prove Conjecture 2
would be enough to get that each faithfully flat ideal in a H-domain
is divisorial. In Example \ref{conjecture2} we have seen that this
is not always true, but we have shown that (faithfully) flat ideals
are
  $t$-ideals (Theorem \ref{thm:tflat}). Now, recall that
H-domains are exactly the domains in which  the $t$-maximal ideals
are all divisorial. Note that if we strengthen this condition
considering domains in which  all the $t$-ideals are divisorial
(\emph{TV-domains}, Cf. \cite{hz}), then for this class of domains
Conjecture 2 is true, since, in this case, flat ideals, being
$t$-ideals, are divisorial. In fact, we prove something more:

\begin{proposition}\label{ff-t-finite character}
Let $D$ be a domain with the $t$-finite character (i.e., each proper
$t$-ideal is contained in finitely many $t$-maximal ideals). Then
each faithfully flat ideal in $D$ is invertible.
\end{proposition}

\begin{proof}
If $I \in \F(D)$ is faithfully flat, then $I$ is locally principal
and, in particular, $I$ is $t$-locally principal (i.e., $ID_P$ is
principal for each $P \in \tmax(D)$). The $t$-finite character of
$D$ implies that $I$ is $t$-finite. Then, by Proposition
\ref{prop:flat_inv}, $I$ is invertible.
\end{proof}

Since TV-domains have the $t$-finite character (\cite[Theorem
1.3]{hz}), we obtain the following:

\begin{corollary}\label{tv} Let $D$ be a TV-domain. Then each
faithfully flat ideal in $D$ is invertible.
\end{corollary}

\begin{remark} Given an integral domain $D$ consider the two
following conditions:

\begin{itemize}
\item[(a)] $D$ has the $t$-finite character;

\item[(b)] each  faithfully flat  ideal in $D$ is invertible.
\end{itemize}

 Proposition \ref{ff-t-finite character} proves that (a) $\Rightarrow$ (b) for any domain $D$.

  We notice that for Pr\"ufer domains   (b) $\Rightarrow$ (a) (in
  this case $t=d$ and (a) is the hypothesis of finite character on
  $D$). This is exactly the content of Bazzoni's conjecture.

Moreover, if $D$ is a Noetherian domain, it is well-known that
 each faithfully flat ideal in $D$ is invertible (see also Propositon \ref{prop:flat_inv}). A
Noetherian domain does not necessarily have the finite character,
but it  does have the $t$-finite character. So, also in this case we
have that (b) $\Rightarrow$ (a).

What we observed for these two relevant classes of domains (the
Pr\"ufer and the Noetherian ones) suggests the following question:

\begin{question}
If each faithfully flat ideal of $D$ is invertible, does $D$ have
the t-finite character?
\end{question}

So far, we are not able to answer to this questions but the
considerations above suggest to investigate in this direction and
try to generalize Bazzoni's conjecture to a class of domains larger than the one of Pr\"ufer domains.
\end{remark}

\section{Quasi-stable domains} \label{sec:quasi-stable}

We recall from the Introduction that a nonzero    ideal $I$ of $D$
is \textit{stable} if $I$ is projective in the endomorphism ring $(I
\colon I)$ and that $D$ is a \textit{stable domain} if each nonzero
ideal of $D$ is stable. Moreover, an integral domain $D$ is
\textit{finitely stable} if each nonzero finitely generated ideal of
$D$ is stable.

 Proposition \ref{ff-t-finite character}   suggests  the
following characterization of stable domains with the $t$-finite
character.

\begin{proposition}\label{t-finite character}
An integral domain $D$ with the $t$-finite character is stable if
and only if each nonzero ideal $I$ of $D$ is faithfully flat in $(I:I)$.
\end{proposition}

\begin{proof}
If $D$ is stable, then each nonzero ideal $I$ of $D$ is invertible
in $(I:I)$, and so $I$ is faithfully flat in $(I:I)$  (and this is
true even without assuming the $t$-finite character).

For the converse, first note that if each nonzero ideal $I$ of $D$
is faithfully flat in $(I:I)$ then, in particular, each finitely
generated ideal $I$ is invertible in $(I:I)$. Thus $D$ is finitely
stable. By
  \cite[Proposition 2.1]{rush}, finitely stable domains have Pr\"ufer
integral closure, whence  all maximal ideals of $D$ are $t$-ideals
by \cite[Lemma 2.1 and Theorem 2.4]{t-linked}. Thus, the $t$-finite
character on $D$ is, in fact, the finite character and, by
\cite[Lemma 3.4]{olb}, all overrings of $D$ have the finite
character.  By  hypothesis, if $I \in \F(D)$, $I$ is faithfully flat
in $(I:I)$ (which has the finite character). So $I$ is invertible by
Proposition \ref{ff-t-finite character} and $D$ is stable.
\end{proof}

\begin{remark} If $D$ does not have the $t$-finite character,
Proposition \ref{t-finite character} does not hold. In fact, take an
almost Dedekind domain $D$ which is not Dedekind (\cite[Example 42.6
 and Remark 42.7]{gilmer}). In this case $t=d$ ($D$ is Pr\"ufer) and $D$
does not have the $t$-finite character. Each ideal of $D$ is locally
principal and so it is faithfully flat in $D$. Moreover, $D$ is the
endomorphism ring of each of its ideals, since it is completely
integrally closed, but $D$ has, at least, a nonzero ideal which is
not invertible and so $D$ is not stable.

After considering the faithfully flat condition on ideals,  it seems
natural to investigate in which domains each nonzero ideal is flat
in its endomorphism ring and compare this new class of domains with
stable and finitely stable domains. \end{remark}

\begin{definition}\label{quasi-stable}
We say that a nonzero \textit{ideal} $I$ of a domain $D$ is
\emph{quasi-stable} if $I$ is flat as an ideal of $(I:I)$ and that a
\textit{domain} $D$ is \emph{quasi-stable} if  each nonzero ideal of
$D$ is quasi-stable.
\end{definition}

\begin{proposition}\label{prop:fs-fqs}
The following conditions are equivalent for an integral domain~$D$:

\begin{enumerate}
\item[(i)] $D$ is finitely stable.
\item[(ii)] Each nonzero finitely generated ideal of $D$ is quasi-stable.
\item[(iii)] For each nonzero finitely generated ideal $I$ of $D$, $I$ is a $t$-ideal of $(I:I)$ and $((I:I):I)$ is a
finitely generated
 ideal of $(I:I)$.
\end{enumerate}
\end{proposition}

\begin{proof}
(i)$\Leftrightarrow$(ii) and (ii)$\Rightarrow$(iii)  are a straightforward consequence of the
fact that finitely generated flat ideals are invertible.

 (iii)$\Rightarrow$(i) follows by applying exactly the same argument used in the proof of
\cite[Theorem 3.5, (ii)$\Rightarrow$(i)]{olb}.
\end{proof}

So, in particular, the Noetherian quasi-stable domains are exactly
the Noetherian stable domains (Cf. \cite[Theorem 11]{goeters}).

Note that since stable ideals are quasi-stable (invertible ideals
are flat), stable domains are quasi-stable. Moreover, it is an easy
consequence of Proposition \ref{prop:fs-fqs} that quasi-stable
domains are finitely stable.

In Example \ref{I is t-ideal in End(I)}, we will show that there
exists an integral domain $R$  that  satisfies condition (iii) of
Proposition \ref{prop:fs-fqs}, but which is not quasi-stable. Thus
we pose the following question:

\medskip
\begin{question} Are the finitely stable domains
the domains in which each ideal (or each finitely generated ideal)
is a $t$-ideal in $(I:I)$? \end{question}

\medskip

This question is also suggested by the following fact. Olberding in
\cite[Theorem 3.5]{olb} has shown that a domain $D$ is stable if and
only if each nonzero ideal $I$ of $D$ is divisorial in its
endomorphism ring $(I \colon I)$.
 Moreover, the $t$-operation is the finite-type operation associated to the $v$-operation and the finitely stable domains are the finite-type version of stable domains. Thus a positive answer to the question above would give a finite-type interpretation of Olberding's result.

\begin{examples} \begin{enumerate}
\item  \emph{A quasi-stable domain that is not stable.}

Each Pr\"ufer domain is quasi-stable, because each ideal of a
Pr\"ufer domain is flat and overrings of Pr\"ufer domains are
Pr\"ufer. Since stable domains have the finite character
(\cite[Theorem 3.3]{olb}), it is enough to take a Pr\"ufer domain
without the finite character (e.g., an almost Dedekind domain which
is not Dedekind) to get an example of a quasi-stable domain which is
not stable.

Note also that  the finite character on $D$ is not sufficient to get
that a quasi-stable domain is stable. Again, a Pr\"ufer domain of
finite character which is not strongly discrete (i.e., it has at
least a prime ideal that is idempotent) is quasi-stable but not
stable (\cite[Theorem 4.6]{olb2}).

\medskip

\item \emph{A quasi-stable non Pr\"ufer domain that is not stable.}

Consider a pseudo-valuation  domain $D$ that is not a valuation
domain with maximal ideal $M$ and associated valuation domain
$M^{-1}=(M:M)=V$ and assume that $V$ is $2$-generated as a
$D$-module. In this case $v=t=d$ on $D$ (\cite[Corollary 1.8]{pvd2}
and \cite[Proposition 4.3]{hz}). So, each ideal of $D$ is principal
or it is a common ideal of $D$ and $V$ (\cite[Proposition
2.14]{PVD1}). If $I$ is principal in $D$, then $(I \colon I) = D$
and $I$ is flat in $D$. So $I$ is quasi-stable. If $I$ is a common
ideal of $D$ and $V$, then $(I:I) \supseteq V$ is a valuation domain
and so $I$ is flat in $(I:I)$. Thus $D$ is quasi-stable.

 If we take $M$ non-principal in $V$,
then $M$ is not invertible in $(M:M)=V$ and $D$ is not stable.
\end{enumerate}
\end{examples}

As we have seen, it is easy to find examples of quasi-stable domains
which are not stable, even in the case of integrally closed domains
with finite character. On the contrary, it seems that quasi-stable
domains are very close to finitely stable domains. We have already
mentioned that these two classes of domains (quasi-stable and
finitely stable) do coincide in the Noetherian case. The next result
shows that they coincide also in  the other classical case  of
integrally closed domains.

\begin{proposition}
Let $D$ be an integrally closed domain. The following conditions are equivalent:
\begin{enumerate}
\item[(i)] $D$ is a quasi-stable domain.

\item[(ii)] $D$ is a finitely stable domain.

\item[(iii)] $D$ is a Pr\"ufer domain.
\end{enumerate}
\end{proposition}
\begin{proof}
(i)$\Rightarrow$(ii) follows from  Proposition \ref{prop:fs-fqs}.

(ii) $\Rightarrow$(iii) follows from \cite[Proposition 2.1]{rush}.

(iii) $\Rightarrow$(i) is obvious.
\end{proof}

Despite of the previous examples, in general finitely stable domains
are not necessarily quasi-stable. The follow-up of this section is
devoted exclusively to the construction of an example of a finitely
stable domain which is not quasi-stable.

\begin{example}\label{finitely stable not flat stable}
\textbf{Example of a domain that is finitely stable but not
quasi-stable.}

Let $\SF_2$  be the field with $2$ elements and $t$ be an
indeterminate over $\SF_2$. Let $(V,M)$ be a DVR with residue field
$\SF_2(t)$: for instance, take  $(V,M) := (\SF_2(t)[[X]],
X\SF_2(t)[[X]]$), and consider the $2$-degree field extension
$\SF_2(t^2) \subsetneq \SF_2(t)$. Let $A := \SF_2[t^2]_{Q}$, where
$Q$ is a nonzero prime ideal of $\SF_2[t^2]$ which does not contain
$t^2$. Then $A$ is a DVR with quotient field $\SF_2(t^2)$. Consider
the following pullback diagram:

$$
\CD
R:=\varphi^{-1}(A) @>>>  A = R/M\\
@VVV @VVV \\
D:=\varphi^{-1}(\SF_2(t^2)) @>>> \SF_2(t^2) = D/M \\
@VVV @VVV \\
V @>{\varphi}>> \SF_2(t) = V/M
\endCD
$$

\medskip

\nt where the horizontal arrows are projections and the vertical
arrows are injections. Now, $D$ is a Noetherian, pseudo-valuation
domain and since $[\SF_2(t):\SF_2(t^2)]=2$, $D$ is totally
divisorial by \cite[Corollary 1.8]{pvd2} and \cite[Proposition
4.3]{hz} (i.e., each ideal of $D$ is divisorial and the same holds
for each overring of $D$). Then $D$ is stable by \cite[Theorem
2.5]{olb}, whence it is finitely stable.

Let $\overline{R}$ denote the integral closure of $R$. Since $R
\subseteq \overline{R} \subset V$, $M$ is a common ideal of $R, \r$
and $V$. Then $A \subseteq \overline{R}/M \subset  \SF_2(t)$ and
$\overline{R}/M$ is the integral closure of $A$ in $\SF_2(t)$
(\cite[Lemme 2]{cahen}), that we denote, as usual, by
$\overline{A}^{\SF_2(t)}$. It  follows immediately that $R \neq
\overline{R}$ because $t \in \overline{A}^{\SF_2(t)} \backslash A$
(whence, the quotient field of $\overline{A}^{\SF_2(t)}$ is
$\SF_2(t)$). It is well-known that $\overline{A}^{\SF_2(t)}$ is the
intersection of the valuation domains   extending $A$ in $\SF_2(t)$
(\cite[Theorem 20.1]{gilmer}) and, by \cite[Corollary 20.3]{gilmer},
the number of these extensions is, at most, the separable degree of
the field extension $\SF_2(t^2) \subset \SF_2(t)$, which is 1.
Hence, $\overline{A}^{\SF_2(t)}$ is simply a DVR (\cite[Theorem
19.16 (d)]{gilmer}).
Thus, $\overline{R} = \varphi^{-1}(\overline{A}^{\SF_2(t)})$ is a
two-dimensional valuation domain in which $M$ is the height-one
prime ideal (\cite[Theorem 2.4]{fontana}). Moreover, the maximal
ideal of $\r$ is principal since $\r/M$ is a DVR, and $\r _M = V$,
which is a DVR, whence the nonzero prime ideals of $\r$ are not
idempotent and $\r$ is totally divisorial (\cite[Proposition
7.6]{bs}).

By \cite[Proposition 3.6]{O} $R$ is finitely stable with principal
maximal ideal $N$. By general properties of pullback constructions,
$R$ is 2-dimensional  with ordered spectrum $(0) \subset M \subset
N$, and $R_M = D$. Since $D$ is 1-dimensional and $R$ is
2-dimensional, $\overline{R}$ does not contain $D$. Moreover, $D$
does not contain $\overline{R}$ because $D$ is not Pr\"ufer and $\r$
does. So $\overline{R}$ and $D$ are not comparable.

\medskip
\textbf{Claim.} Each ring   between $R$ and $V$ is comparable with
$D$ or $\overline{R}$. First notice that $M$ is a common ideal of
all rings  between $R$ and $V$. Let $B$ be such a ring and suppose
that $B$ is not comparable with $D$. Then $ B/M \not\subseteq
\SF_2(t^2)$ (since $D = \varphi^{-1}(\SF_2(t^2))$). But $A \subset
B/M$ (because $R \subset B$), so $\overline{A}^{\SF_2(t)} \subseteq
\overline{B/M}^{\SF_2(t)}$. As $\overline{A}^{\SF_2(t)}$ being a
DVR, it follows that $\overline{B/M}^{\SF_2(t)} =
\overline{A}^{\SF_2(t)}$ or $\overline{B/M}^{\SF_2(t)} = \SF_2(t)$.
In the first case, we have that $B/M \subseteq
\overline{A}^{\SF_2(t)}$ and so $B \subseteq \r$ (recall that $\r =
\varphi^{-1}(\overline{A}^{\SF_2(t)})$). The second case occurs if
and only if $B/M = \SF_2(t)$ and so $B=V$, which contains $\r$.

\medskip
By \cite[Theorem 4.11]{olb},   $R$ is not stable because $R_M = D$
is not a   valuation domain. Hence there exists a nonzero ideal
 in $R$ which is not divisorial in $(I \colon I)$
(\cite[Theorem 3.5]{olb}). Our aim is to show that this specific
ideal $I$ is not flat in $\ii$ and so $R$ is not quasi-stable.

If $I$ is finitely generated, then $I$ is stable since $R$ is
finitely stable and so $I$ is divisorial in $(I \colon I)$.

Then we can suppose that $I$ is not finitely generated and we
distinguish the following cases:
\begin{itemize}
\item[(a)] $\ii = R$;

\smallskip

\item[(b)] $\ii \neq R$  and $\ii$ is comparable with $D$;

\smallskip
\item[(c)] $\ii \neq R$  and $\ii$ is comparable with $\r$.
\end{itemize}

\medskip

(a) If $\ii = R$ and $I$ is flat in $R$, then $I$ is principal or $I
= IN$ by \cite[Lemma 2.1]{flat1}. We are supposing that $I$ is not
finitely generated, so $I=IN$. But $N= \pi R$ is principal and
$I=I\pi$ implies that $\pi, \pi^{-1} \in \ii = R$, which is
impossible. So in this case $I$ is not  flat in $\ii$ and $R$ is not
quasi-stable.

\smallskip
(b) If $\ii \neq R$  and $\ii$ is comparable with $D$, then $D
\subseteq \ii$ because  between $R$ and $D$ there are no  domains,
since there are no  domains between $A$ and $\SF_2(t^2)$ (because
$A$ is a DVR). But $D$ is totally divisorial, whence $I$ would be
divisorial in $\ii$ against the assumption. Thus, this case cannot
occur.

\smallskip
(c) If $\ii \neq R$  and $\ii$ is comparable with $\r$, then $\r
\subseteq \ii$ or $\ii \subsetneq \r$. In the first case, since $\r$
is totally divisorial,  $I$ would be divisorial in $\ii$, against
the assumption. So, we can assume that $\ii \subsetneq \r$. Then $A
\subsetneq \ii/M \subsetneq \overline{A}^{\SF_2(t)}$. So $\ii/M$ is
local (since its integral closure is $\overline{A}^{\SF_2(t)}$),  it
is Noetherian (by Krull-Akizuki Theorem) and it is not a PID. In
fact,   $\ii/M$ is not integrally closed (since it is strictly in
between $A$ and $\overline{A}^{\SF_2(t)}$). It follows that $\ii$ is
two-dimensional, with prime spectrum $(0) \subsetneq M \subsetneq
\m$ and $\m$ is   not principal (since $\m/M$ is not principal). If
$I$ is flat in $\ii$, then $I$ is principal or $I\m=I$ (again by
\cite[Lemma 2.1]{flat1}). Since $I$ is supposed to be not divisorial
in $\ii$, we have that $I\m=I$. Thus, $(\m \colon \m) \subseteq (I\m
\colon I\m) = \ii$, and so $(\m \colon \m) = \ii$. But $\m$ is not
principal and $\m^2 \neq \m$, since $\m/M$ is not idempotent, as
being $\m/M$ finitely generated. Then $\m$ is not flat in $\ii = (\m
\colon \m)$. We finally notice that $\ii$ is an overring of $R$,
which is finitely stable, whence $\ii$ is finitely stable. Thus, in
this case,$\ii$ is an example of finitely stable domain, which is
not quasi-stable.

We remark  that, from a result that we will prove in the next
section (Proposition~\ref{nonzero conductor}), we also have that
$\ii$ non quasi-stable implies that $R$ is not quasi-stable too.

\end{example}

\begin{example}\label{I is t-ideal in End(I)}
Consider the domain $R$ constructed in the example above. We have
seen that $R$ is finitely stable but  not quasi-stable. We now show
that each nonzero ideal $I$ of $R$ is a $t$-ideal in $(I \colon I)$.

Without loss of generality we can consider only integral ideals.

By construction, each integral ideal $I$ of $R$ is comparable with
$M$.

Suppose that $M \subsetneq I$, then $I = \pi^sR$ is principal, thus
it is a $t$-ideal (recall that the maximal ideal of $R$ is $N = \pi
R$ and $R/M$ is a DVR).

Conversely, let $I \subseteq M$. We consider two sub-cases:

\begin{enumerate}

\item[(a)] The domain $(I \colon I)$ is comparable with $D$.

\nt If $D \subseteq (I \colon I)$, then $(I \colon I)$ is a
divisorial domain (since $D$ is totally divisorial) and so each
ideal of $(I \colon I)$ is a $t$-ideal.

\nt If $R \subseteq (I \colon I) \subsetneq D$, then $(I \colon
I)=R$ and $I$ is $M$-primary in $R$. By \cite[Proposition 4.8]{AM},
$IR_M \cap R = I$. But $R_M=D$, $ID$ is a $t$-ideal in $D$, so $I$
is a $t$-ideal in $R$.

\item[(b)] The domain $(I \colon I)$ is comparable with $\r$.

\nt If $\r \subseteq (I \colon I)$, then $(I \colon I)$ is a
Pr\"ufer domain and so each ideal is a $t$-ideal.

\nt If $R \subsetneq (I \colon I) \subsetneq \r$, then    the
quotient field of $(I \colon I)/M$ is $\SZ_2(t)$. Then $(I \colon
I)_M=V$, $I$ is $M$-primary in $(I \colon I)$, $IV$ is a $t$-ideal
and so $I$ is a $t$-ideal by the same argument used above.

\end{enumerate}
\end{example}

\section{Overrings of quasi-stable domains} \label{sec:overrings}

It is known that overrings of stable domains are stable
 and overrings of finitely
stable domains are finitely stable (\cite[Theorem 5.1 and Lemma
2.4]{olb}). In this section we study the quasi-stability for
overrings of quasi-stable domains. We are able to prove that
overrings of quasi-stable domains are still quasi-stable for some
relevant classes of overrings (a general result is given in
Corollary~\ref{overrings}).

 The first result of this section is a generalization of the flatness criterion
 for ideals in integral domains  recalled in Proposition \ref{char:flatness}.

\begin{proposition}\label{char:flatness_modules}
Let $D$ be an integral domain and $I$ be a nonzero ideal of $D$.
Then $I$ is flat over $D$ if and only if $I(A \cap B) = IA \cap IB$,
for all $A,B$   $D$-submodules of $K$.
\end{proposition}

\begin{proof}
The ``if" part is already shown in Proposition \ref{char:flatness}
since ideals are, in particular, $D$-submodules of $K$.

So we will prove  the ``only if" part. It is well-known
(\cite[Theorem 7.4]{matsumura}) that if $I$ is a flat $D$-module and
$A,B$ are $D$-submodules of $K$, then $I\otimes_D(A \cap B) = (I
\otimes_D A) \cap (I \otimes_D B)$. So it is enough to show that $I
\otimes_D N \cong IN$ for each $D$-submodule $N$ of $K$.

Consider the following surjective homomorphism of $D$-modules:
$$\varphi: I \otimes_D N \twoheadrightarrow IN, \quad i \otimes_D n
\mapsto in.$$

We show that $\varphi$ is injective, so obtaining that $I \otimes_D
N \cong IN$. Consider the exact sequence:
$$0 \rightarrow N \rightarrow K.$$

\nt For the $D$-flatness of $I$, the sequence $0 \rightarrow I
\otimes_D N \rightarrow I \otimes_D K$ is exact.

Suppose that $\varphi(\sum_{j=1}^si_j \otimes_D n_j) = \sum_{j=1}^si_j
n_j = 0$. Then
$$0 = \sum_{j=1}^si_jn_j \otimes_D 1_D = \sum_{j=1}^si_j
\otimes_D n_j   \in I \otimes_D K.$$
Thus $\sum_{j=1}^si_j \otimes_D n_j
= 0 \in I \otimes_D N$ for the exactness of the sequence above.

 This completes the proof.
\end{proof}

\begin{proposition}\label{flatness-overring}
Let $D$ be an integral domain and $I$ be a nonzero ideal of $D$.
  Let $T$ be an overring of $D$. If $I$ is a flat ideal of   $D$
  then $IT$ is a flat ideal of $T$.

\end{proposition}

\begin{proof}
It is enough to observe that the $T$-submodules of $K$ are also
$D$-submodules of $K$ and apply Proposition
\ref{char:flatness_modules}.
\end{proof}

\begin{corollary}\label{prop:flat-conductor}
Let $D$ be an integral domain and $I$ be a nonzero ideal of $D$.
\begin{enumerate}

\item[(a)] If $I$ is flat, then $I$ is quasi-stable.

\item[(b)] If $I$ is a flat ideal of $D$, then $I$ is a $t$-ideal of $(I:I)$.
\end{enumerate}
\end{corollary}

\begin{proof}

(a) It is immediate from Proposition \ref{flatness-overring}, since
$(I:I)$ is an overring of $D$ and $I=I(I \colon I)$.

(b) It follows from (a) and Theorem \ref{thm:tflat}.
\end{proof}


We recall the following result due to D. Rush (\cite[Proposition 2.1]{rush}).

\begin{proposition}
Let $D$ be a  finitely stable domain. Then the integral closure
$\overline{D}$ of $D$ is  a Pr\"ufer domain.
\end{proposition}

Since quasi-stable domains are finitely stable we have the following
corollary:

\begin{corollary}
The integral closure of a quasi-stable domain is a Pr\"ufer domain
and so it is quasi-stable.
\end{corollary}

\begin{proposition}\label{quasi-flatnes in overring}
Let $D$ be an integral domain and $T$  be an overring of $D$. If $I$
is a quasi-stable ideal of $D$, then $IT$ is a quasi-stable ideal of
$T$.
\end{proposition}

\begin{proof}
Since $I$ is flat in $(I \colon I)$, then $IT = I(I \colon I)T$ is
flat in $(I \colon I)T$,   by Proposition \ref{flatness-overring}.
Now, $(I \colon I)T \subseteq (IT \colon IT)$,   so  applying  again
  Proposition \ref{flatness-overring},   we obtain that   $IT$
is a flat ideal of $(IT \colon IT)$.
\end{proof}

As it is stated in the next result, a  case in which the
quasi-stability transfers to overrings is when we have a ring
extension $D \hookrightarrow T$ such that map
$$\Phi_D^T: \F(D) \rightarrow \F(T), \quad I \mapsto IT$$
\nt is surjective that is, when each ideal of $T$ is an extension of
an ideal of $D$ (we remark that this includes also the case when an
integral ideal of $T$ is an extension of a fractional ideal of $D$).

\begin{corollary}\label{overrings}
Let $D$ be an integral domain and let $T$ be an overring of $D$ such
that $\Phi_D^T$ is surjective. Then, if $D$ is quasi-stable, $T$ is
quasi-stable.
\end{corollary}

\begin{proof}
It is an immediate consequence of Proposition \ref{quasi-flatnes in
overring}.
\end{proof}

Interesting classes of overrings of a   domain $D$   which
satisfy the condition of Corollary \ref{overrings} are studied in
\cite{sega} and we list them  as follows:

\begin{itemize}
\item $T$ is an overring of $D$ such that
$(D:T) \neq 0$ (a particular case is when $T = (I \colon I)$);

\item $T$ is a flat overring of $D$ (i.e., $T$ is flat as a
$D$-module);

\item $T$ is a Noetherian overring of $D$;

\item $T$ is \textit{well-centered} on $D$ (i.e., for all $t \in T$ there exists $u \in U(D)$ such that $ut \in
D$);

\item $T$ is any overring of a domain $D$ such that   $\overline{D}$   is  Pr\"ufer
 and it is a
(fractional) ideal of $D$.
\end{itemize}

Recalling that if $D$ is quasi-stable then $D$ is finitely stable
and so its integral closure is Pr\"ufer, from the last point of the
list above we get the following:

\begin{corollary}\label{nonzero conductor}
Let $D$ be an integral domain such that $(D:\overline{D}) \neq (0)$.
If $D$ is quasi-stable, then every overring of $D$ is quasi-stable
\end{corollary}


A domain $D$ is called \emph{conducive} if $(D:T) \neq (0)$  for all
overrings of $D$.

\begin{corollary}\label{conducive}
An overring of a conducive quasi-stable domain is quasi-stable.
\end{corollary}

Note that there exist quasi-stable domains which are not conducive
(for example, not all Pr\"ufer domains are conducive).

\bigskip

The study of stability and finite stability can be reduced to the
local case, since a domain is stable if and only if it is locally
stable and it has the finite character (\cite[Theorem 3.3]{olb}),
and it is finitely stable if and only if it is locally finitely
stable. We approach this question  in the case of quasi-stable
domains.

Any localization of a domain $D$ is a flat overring of $D$. Thus, we
can easily get the following result as a corollary of
Corollary~\ref{overrings}.

\begin{corollary}
A quasi-stable domain $D$ is locally quasi-stable (i.e., $D_P$ is
quasi-stable for each $P \in \spec(D)$).
\end{corollary}
For the inverse implication, that is whether a locally quasi-stable
domain is quasi-stable, we give partial results.

We recall that a domain $D$ is $h$-local if each nonzero ideal $I$
of $D$ is contained in at most finitely many maximal ideals of $D$
and each nonzero prime ideal of $D$ is contained in a unique maximal
ideal of $D$. Examples of $h$-local domains are   one-dimensional
Noetherian domains or domains in which each nonzero ideal is
divisorial (\cite{heinzer,olb2}.

We show that if a domain $D$ is locally quasi-stable and $h$-local,
then $D$ is quasi-stable. Note that this does not allow us to reduce
the problem of   flat-stability
 to the local case, because quasi-stable domains are not necessarily $h$-local  (a Pr\"ufer domain is quasi-stable but it may not be $h$-local).

\begin{lemma}
Let $D$ be an integral domain and $I$ a nonzero ideal of $D$. Assume
that $(I:I)D_M = (ID_M:ID_M)$, for all $M \in \max(D)$.  If $ID_M$
is quasi-stable (as an ideal of $D_M$) for all $M \in \max(D)$, then
$I$ is  quasi-stable.
\end{lemma}

\begin{proof} We need to show that $I(A \cap B) = IA \cap IB$, for
each $A,B \in \F((I \colon I))$. This is equivalent to showing that
$I_M(A_M \cap B_M) = I_MA_M \cap I_MB_M$, for each $M \in \max(D)$.
But $A_M,B_M \in \FF((I \colon I)_M)$ and since, by hypothesis
$(I:I)D_M = (ID_M:ID_M)$, $A_M,B_M$ are $(ID_M:ID_M)$-modules. So
$I_M(A_M \cap B_M) = I_MA_M \cap I_MB_M$ because $ID_M$ is flat over
$(ID_M:ID_M)$.
\end{proof}

Note that the equality $(I:I)D_M=(ID_M:ID_M)$ is always satisfied
when $I$ is finitely generated, by the flatness of $D_M$ over $D$.
But this case is not interesting since quasi-stable  finitely
generated ideals are  stable (and have already been widely studied
especially in the finitely generated case, Cf. \cite{goeters,rush}).

In general, as the following example shows, it may happen that
$(I:I)D_M \neq (ID_M:ID_M)$ even in quasi-stable domains.

\begin{example}
Consider the domain Int$(\SZ) := \{f(X) \in \SQ[X] \mid f(\SZ)
\subseteq \SZ \}$. It is well-known that Int$(\SZ)$ is completely
integrally closed,  being $\SZ$ completely integrally closed
(\cite[Proposition VI.2.1]{cc}). Thus, $(I \colon I) = \intz$, for
each nonzero ideal $I$ of $\intz$. It is also well-known that
$\intz$ is a two-dimensional Pr\"ufer domain (\cite{cc}), whence
there exists a maximal ideal $M$ such that $\intz_M$ is a
two-dimensional valuation domain. It follows that $\intz_M$ is not
completely integrally closed and so there exists a nonzero ideal $I$
of $\intz$ such that $(I_M \colon I_M) \neq \intz_M$. But, $\intz  =
(I \colon I)$, so we have that $(I_M \colon I_M) \neq (I \colon
I)_M$.
\end{example}

Olberding (\cite[Lemma 3.8]{olb2}) has shown that if $D$ is
$h$-local, then the equality $(I \colon I)D_M = (ID_M \colon ID_M)$
holds, for each $I \in \F(D)$ and $M \in \max(D)$. Then, for
$h$-local domains,   the quasi-stable property can be locally
verified.

\begin{corollary}
Let $D$ be an $h$-local domain.   Then $D$ is quasi-stable if and
only if $D_M$ is quasi-stable for each $M \in \max(D)$.
\end{corollary}

\ec




\bigskip

\begin{thebibliography}{12}

\bibitem{aa} {\sc D.D. Anderson and D.F. Anderson}, {\em Generalized GCD domains}, Comment. Math.
Univ. St. Paul.  {\bf 28}  (1980), no. 2, 215--221.

\bibitem{aa2} {\sc D.D. Anderson and D.F. Anderson}, {\em Some remarks on cancellation ideals. Math.
Japon},  {\bf 29} no. 6 (1984),  879--886.

\bibitem{and1} {\sc D.D. Anderson}, {\em On the ideal equation $I(B\cap C)=IB\cap IC$}, Canad. Math.
Bull.  {\bf 26} no. 3 (1983),  331--332.

\bibitem{AM} {\sc M.F. Atiyah and I.G. Macdonald},  {\em Introduction to commutative algebra},
Addison-Wesley Publishing Co., Reading, Mass.-London-Don Mills, Ont. 1969

\bibitem{bazzoni} {\sc S. Bazzoni}, {\em Class semigroups of Pr\"ufer domains}, J. Algebra
{\bf 184} (1996),  613–-631.

\bibitem{bs} {\sc S. Bazzoni and L. Salce}, {\em Warfield domains}, J. Algebra
  {\bf 185} no. 3 (1996), 836--868.

\bibitem{baz1} {\sc S. Bazzoni},  {\em Groups in the class semigroup of a Pr\"ufer domain of
finite character}, Comm. Algebra {\bf 28} (2000),  5157--5167.

\bibitem{baz2} {\sc S. Bazzoni}, {\em Clifford regular domains}, J. Algebra {\bf 238} (2001),
703--722.



\bibitem{cahen} {\sc P.J. Cahen}, {\em Couples d'anneaux partageant un
id\`eal}, Arch. Math. (Basel) {\bf 51} no. 6 (1988),  505--514.

\bibitem{cc}
{\sc P.-J. Cahen and J.-L. Chabert}, {\em Integer-Valued
Polynomials}, Amer.\ Math.\ Soc.\ Surveys and Monographs, {\bf 48},
Providence, 1997.

\bibitem{heinzer}
{\sc W. Heinzer}, {\em Integral domains in which each nonzero ideal
is divisorial}, Mathematika  {\bf 15} (1968), 164--170.


\bibitem{t-linked}
{\sc D.E. Dobbs,  E.G.  Houston,  T.G. Lucas,  M. Roitman and M. Zafrullah},
{\em On $t$-linked overrings}, Comm. Algebra {\bf 20} no. 5 (1992),
1463--1488.

\bibitem{fontana}{\sc M. Fontana},
{\em Topologically Defined Classes of Commutative Rings}, Annali di
Matematica Pura ed Applicata {\bf 123} (IV)( 1980), 331--355.

\bibitem{fh}{\sc M. Fontana and J. Huckaba}, {\em Localizing systems and semistar
operations}, Non-Noetherian commutative ring theory, 169--197, Math.
Appl., 520, Kluwer Acad. Publ., Dordrecht, 2000.

\bibitem{gabelli} {\sc S. Gabelli}, {\em Generalized Dedekind domains},
Multiplicative ideal theory in commutative algebra, 189--206,
Springer, New York, 2006.

\bibitem{gilmer} {\sc R. Gilmer}, {\em Multiplicative Ideal
Theory}, Marcel Dekker, New York, 1972; rep. Queen's Papers in Pure
and Applied Mathematics, {\bf vol.\ 90}, Queen's University,
Kingston, 1992.


\bibitem{GV}{\sc S. Glaz and W.V. Vasconcelos},
{\em Flat ideals. II}, Manuscripta Math. {\bf 22} no.4 (1977),
325--341.

\bibitem{flat3}{\sc S. Glaz and W.V. Vasconcelos},
{\em Flat ideals. III}, Comm. Algebra {\bf 12} no.2 (1984),
199--227.



\bibitem{goeters} {\sc H.P. Goeters}, {\em Noetherian stable domains}, J. Algebra  {\bf 202}  (1998), no. 1, 343--356.


\bibitem{hk} {\sc F. Halter-Koch}, {\em  Clifford Semigroups of Ideals in Monoids and
Domains}, Forum Math. (to appear).

\bibitem{PVD1} {\sc J.R. Hedstrom and E.G.  Houston}, {\em Pseudo-valuation domains}, Pacific J. Math.  \textbf{75}  (1978), no. 1, 137--147.

\bibitem{pvd2} {\sc J.R. Hedstrom and E.G.  Houston},   {\em Pseudo-valuation domains. II}, Houston J. Math.  \textbf{4}  (1978), no. 2, 199--207.

\bibitem{hmmt} {\sc W.C. Holland,  J. Martinez,  W.Wm. McGovern  and M.
Tesemma}, {\em Bazzoni's conjecture},  J. Algebra \textbf{320}
(2008), no. 4, 1764--1768 .

\bibitem{hz}{\sc E.G. Houston and M. Zafrullah},
{\em Integral domains in which each $t$-ideal is divisorial}, J.
Michigan Math. J. {\bf 35} no.2 (1988), 291--300.

\bibitem{km} {\sc S.E. Kabbaj and A. Mimouni}, {\em Class semigroups of integral domains}, J.
Algebra {\bf 264} (2003), 620--640.

\bibitem{km2} {\sc S.E. Kabbaj and A. Mimouni}, {\em $t$-class semigroups of integral
domains}, J. Reine Angew. Math. {\bf 612} (2007), 213--229.

\bibitem{matsumura} {\sc H. Matsumura}, {\em Commutative Ring Theory}, Cambridge University Press, 1989.

\bibitem{om}
{\sc A. Okabe and R. Matsuda}, {\em Semistar-operations on integral
domains}, Math. J. Toyama Univ. \textbf{17} (1994), 1--21.


\bibitem{olb2} {\sc B. Olberding}, {\em Globalizing local properties of Pr\"ufer domains}, J. Algebra  {\bf 205} no. 2  (1998),   480--504.

\bibitem{O}{\sc B. Olberding},
{\em On the classification of stable domains}, J. Algebra {\bf 243}
(2001), 177--197.

\bibitem{olb}{\sc B. Olberding},
{\em On the structure of stable domains}, Comm. Algebra {\bf 30} no.
2 (2002), 877--895.

\bibitem{popescu} {\sc N. Popescu},  {\em On a class of Pr\"ufer domains}, Rev. Roumaine
Math. Pures Appl. \textbf{29} no. 9 (1984),  777--786.

\bibitem{rush} {\sc D.E. Rush}, {\em Two-generated ideals and representations of abelian groups over valuation rings},
 J. Algebra  {\bf 177} no. 1  (1995), 77--101.

\bibitem{flat1}{\sc J.D. Sally and  W.V. Vasconcelos},
{\em Flat Ideals I}, Comm. Algebra {\bf 6} no.3 (1975), 531--543.

\bibitem{sega} {\sc L. Sega}, {\em Ideal class semigroups of overrings}, J. Algebra {\bf 311} no. 2 (2007), 702--713.

\bibitem{zaf}{\sc M. Zafrullah}, {\em Well-behaved prime $t$-ideals},
J. Pure Appl. Algebra  {\bf 65} no. 2 (1990), 199--207.
%

\bibitem{zaflat} {\sc M. Zafrullah}, {\em Flatness and invertibility of an ideal}, Comm. Algebra  {\bf 18}  no. 7 (1990),  2151--2158.

\bibitem{zariski-samuel} {\sc O. Zariski and P. Samuel}, {\em Commutative algebra Vol. II},  Graduate Texts in Mathematics, Vol. 29.
Springer-Verlag, New York-Heidelberg, (1975).
\end{thebibliography}
\end{document}